\documentclass[11pt]{amsart}
\newtheorem{theorem}{Theorem}[section]
\newtheorem{lemma}[theorem]{Lemma}

\theoremstyle{definition}

\newtheorem{algorithm}[theorem]{Algorithm}
\newtheorem{example}[theorem]{Example}

\theoremstyle{remark}

\newcommand{\R}{{\mathbb R}}
\newcommand{\I}{{\mathbb I}}

\newcommand{\cI}{\mathcal I}
\newcommand{\cJ}{\mathcal J}
\newcommand{\cK}{\mathcal K}
\newcommand{\cC}{\mathcal C}

\begin{document}

\title[Sparse Null Vector Problem]{Random Search Algorithms for the Sparse Null Vector Problem}

\author{Hans Engler}
\address{Dept. of Mathematics \\
Georgetown University\\
Box 571233\\
Washington, DC 20057\\
USA} \email{engler@georgetown.edu}

\begin{abstract}We consider the following problem: Given a matrix $A$, find minimal subsets of columns of $A$ with cardinality no larger than a given bound that are linear dependent or nearly so. This problem arises in various forms in optimization, electrical engineering, and statistics. In its full generality, the problem is known to be NP-complete. We present a \textsc{Monte Carlo} method that finds such subsets with high confidence. We also give a deterministic method that is capable of proving that no subsets of linearly dependent columns up to a certain cardinality exist. The performance of both methods is analyzed and illustrated with numerical experiments.
\end{abstract}

\maketitle

\section{Introduction}

Let $A$ be a real $M \times N$ matrix with rank $m < N$. In this paper, the following problem is studied:

\smallskip
\emph{Find a minimal set of linearly dependent columns of $A$,}

\smallskip
that is, a set of linearly dependent columns of $A$ such that each proper subset consists of linearly independent columns. An essentially equivalent problem is:

\smallskip
\emph{Find a null vector $x$ of $A$ whose support $\{i \, | \, x_i \ne 0\}$ is minimal.}

\smallskip
This problem is commonly called the \emph{sparse null vector problem}, and an extension of the problem, known as the \emph{sparse null space problem}, is:

\smallskip
\emph{Find a basis of the null space of $A$ whose vectors all have minimal support}.

\smallskip
The sparse null space problem occurs in optimization and in finite element analysis, where it is often of interest to express all solutions of an underdetermined system of equations $Ax = b$ (a system of constraint equations or a discrete balance law) in the form $x = x_0 + Cu$, with a matrix $C$ whose columns span the null space of $A$. Sparsity of $C$ leads to well-known computational advantages. The sparse null space problem is discussed in detail in \cite{coleman_pothen1} and \cite{coleman_pothen2}, and approximative algorithms for its solution are presented. Further work may be found in \cite{brualdi}, \cite{gilbert_heath}, \cite{pinar}, \cite{stern}. The sparse null space problem can be solved with a canonical greedy algorithm that looks for a sequence of sparsest linearly independent null vectors. On the other hand, it is known that the sparse null vector problem is NP-complete in its full generality. Exceptions are also known. For example, a polynomial algorithm is known if $A$ is the vertex-edge incidence matrix of a graph, see \cite{itai}.

\smallskip
The problem of identifying minimal sets of linearly dependent columns of a matrix occurs also in electrical engineering. Here, the goal is to identify the behavior of the components of a circuit from measurements at a set of test points (input frequencies for an analog circuit, test words for a digital-to-analog converter, physical test nodes). Then it may happen that faults in a group of components are indistinguishable from one another. A simple example is a string of electric light bulbs: If one of the bulbs is defective, the entire string is dark, and it is not immediate which light has to be replaced. The connection with linear algebra comes up as follows. Small deviations from the nominal circuit behavior can be described by a matrix equation $Ax=b$ for the linearized response of a circuit near a desired behavior. The vector $b$ encodes deviations from nominal measurements at various test points, and the unknown $x$ corresponds to deviations of parameters for circuit components from their desired status. If a subset of columns of $A$ is linearly dependent, then the corresponding components of $x$ are not unique, but satisfy some affine relation. Thus the deviations in the behavior of components in this group cannot be determined uniquely, and faults cannot be located unambiguously. Such groups of components are called \emph{ambiguity groups} in the engineering literature, and minimal groups are called \emph{canonical ambiguity groups}.  Their identification is useful to guide design modifications to improve the testability of a circuit. A discussion of this problem and of the  numerical issues associated with it may be found in \cite{stenbak}. This paper also gives an algorithm (essentially a complete search) that can lead to the identification of all canonical ambiguity groups of moderate size. Further algorithmic approaches were presented in \cite{fedi}, \cite{manetti}, and \cite{starzyk}.

\smallskip
A common framework for this problem is provided by matroid theory; see e.g. \cite{oxley} for an introduction. Here, subsets of elements of an abstract base set may be independent or dependent, and it is of interest to determine minimal dependent sets. Concrete examples for dependent sets are linearly dependent sets of vectors or vertices on a closed path in a graph. Minimal dependent sets are called \emph{circuits} in this theory and correspond to closed paths with no repeated vertices in the graph theory context. The problem of enumerating all circuits up to a given size in a matroid is discussed in \cite{boros}. The structure of the set of all circuits up to a given size can be very complex.

\smallskip
In recent work on \emph{compressed sensing} or \emph{compressive sampling}, it has been discovered that the sparsest solution $x$ (in the sense of having the fewest number of non-zero components) of an underdetermined system of equations $A_0x = b$ may often be recovered exactly or approximately by looking for a solution with minimal $l^1$-norm; see e.g. \cite{candes} and \cite{donoho}. An otherwise intractable problem can therefore be attacked with linear programming methods. It appears therefore to be promising to find minimal sets of linearly dependent columns of $A$ by removing a single column (call this column $b$ and call the remaining matrix $A_0$) and looking for a sparse solution of the system $A_0x = b$, using $l^1$-minimization. A sparse solution of this problem immediately results in a sparse null vector. Repeating this for all columns of $A$ would give all sparse null vectors. However, a typical assumption in these results, known as \emph{restricted isometry property}, is that the condition numbers of all $M \times K$ submatrices of $A_0$ have to be uniformly bounded, for some sufficiently large $K$. It is easy to see that in this case $A$ cannot have many disjoint linearly dependent subsets of size $\le K$. Indeed, for coefficient matrices with many disjoint subsets of linearly dependent columns, $l^1$-minimization may fail to detect the sparsest solution of a system of linear equations. We illustrate this with the following example.

\begin{example} Let $B, C$ be $M \times L$ matrices with rank $L$ such that the column spaces have trivial intersection (hence $2L \le M$). For $\beta, \, \gamma \ne 0$, consider the matrix
$$A = (B \quad C \quad \beta B + \gamma C)\, . $$
This matrix has many sets of minimal linearly dependent columns of size 3. Specifically, denoting by $\mathbf{a}_i$ the $i$-th column of $A$, all sets of the form $\{\mathbf{a}_i,\, \mathbf{a}_{i+L}, \, \mathbf{a}_{i+2L}\}, \, 1 \le i \le L$ have this property. Let $0 \ne y\in \R^L$ be arbitrary, and set $b = By$. The sparsest solution of the equation $Ax = b$ is $x = (y^T, \, 0, \, 0)^T$. However, if $\frac{|\gamma|}{|\beta|} + |1 + \gamma| < 1$, then the $l^1$-minimal solution of this problem is $\tilde x =(0, \, -\frac{\gamma}{\beta} y^T, \, -(1+\gamma) y^T)^T$ which has twice as many non-vanishing entries as $x$ if $\gamma \ne -1$. That is, if $-1 < \gamma < 0$ and $|\beta| > 1$ or if $-2 < \gamma < -1$ and $|\beta| > \frac{|\gamma|}{2+\gamma}$, the $l^1$-minimal solution is not the sparsest solution.
\end{example}
In a sense, this paper is concerned with methods for examining the residual class of matrices for which $l^1$-minimal solutions of overdetermined linear systems are not necessarily sparsest.

\smallskip
We finally mention some problems in statistics that are related to the topic of this paper. Consider a linear regression problem $Ax \approx b$, where the columns of $A$ (the predictors) are now assumed to be linearly independent. The least squares solution is $\hat{x} = \left(A^TA\right)^{-1}A^T b$. If $A$ has small singular values, then the estimate $\hat x$ will depend sensitively on small changes in the data $b$ and may in fact be nonsensical. Often this instability occurs because a group of columns of $A$ is nearly linearly dependent, a phenomenon known as \emph{multicollinearity} (see e.g. \cite{kutner}). This happens e.g. when similar measures of the same phenomenon are included in the set of predictors. While multicollinearity may be addressed by judicious choices of the variables that are included in a regression model, the detection of a subset of closely related variables in a data set is often interesting in its own right. A similar problem is the selection of a small set of variables in a regression model that may be used to predict a response. To place this in the present setup, consider the augmented matrix $\tilde A = (A \quad  b)$. One would like to find a minimal set of nearly linearly dependent columns of $\tilde A$ that contains $b$. This specific model selection problem is attacked with machine learning techniques in \cite{neylon}. More broadly, the problem of detecting sets of components that are nearly linearly related in a high-dimensional data set belongs in the area known as \emph{association mining}, see e.g. \cite{ceglar}.

\smallskip
In this paper, two algorithms are discussed that may be used to detect or rule out the existence of small minimal subsets of columns of an $M \times N$ matrix $A$ of rank $m$ that are exactly or nearly linearly dependent. Suppose a single such subset of size $n$ is present. The first method given here, a random search, is expected to detect it with probability $1 - \epsilon$ by examining about $ |\log \epsilon \cdot \rho^{-n}|$ submatrices, where $\rho = \frac{m}{N}$. The submatrices are expected to have about $m(1 - \rho)$ rows and columns each. The second method, a systematic search, is capable of ruling out the presence of such a subset by examining about $\binom{\lceil n/\rho \rceil}{n}$ submatrices of similar size, although recursive calls of the search routine may increase the computational effort. In either case, matrices with small relative rank defects, i.e. $\rho$ close to 1, offer the best chances to detect or rule out the presence of small sets of minimal linearly dependent columns. We also give a modification of the random search method for the problem of finding minimal subsets of columns that are nearly linearly dependent.

\smallskip
The paper is organized as follows. Section 2 contains definitions and basic facts. In section 3, two versions of the random search method for exactly dependent subsets are introduced. The systematic search is presented in section 4. Section 5 contains the modification of the random search method for the case of nearly dependent sets of columns. In section 6, results from numerical experiments are presented. In section 7, we briefly discuss some related problems.

\smallskip
The author would like to thank G. Stenbakken and T. Souders for introducing him to the problem and for many stimulating discussion.

\section{Notation and Auxiliary Results}

Given two nonnegative integers $a \le b$, we denote the set $\{a, a+1,
\dots, b\}$ by $\overline{a,b}$ and identify it with the vector $(a, \, a+1, \dots, b)$. Given $N \ge 1$, let $\cI = \{i_1, i_2,
\dots, i_n\} \subset \overline{1,N}$ with $i_1 < i_2 < \dots$.
We identify $\cI$ with the vector $(i_1, \dots, i_n)$ and write $i_1 = \cI(1),i_2 = \cI(2), \dots$.

\smallskip
Let $A$ be a real $M \times N$ matrix. Using \textsc{Matlab} notation, for a non-empty set $\cI \subset
\overline{1,N}$, we write $A(:,\cI)$ for the $M \times |\cI|$ matrix
obtained by extracting all columns with indices in $\cI$, with the
ordering of the rows remaining the same. Similarly, for a set $\emptyset \ne \cJ \subset
\overline{1,M}$, $A(\cJ,:)$ is the $|\cJ| \times N$ matrix obtained by
extracting all rows whose indices are in $\cJ$. Then $A(\cJ,\cI)$ is
the submatrix obtained from $A$ by extracting the rows with
indices in $\cJ$ and columns with indices in $\cI$. If $y$ is a row or
column vector, $y(\cI)$ is the row or column vector with the components
indexed by $\cI$ extracted from those of $y$. The \emph{support} of the vector $y$ is defined as $supp(y) = \{i \, |\, y_i \ne 0\}$.  We denote the $m \times m$ identity matrix by $\I_m$.

\smallskip
Assume now that $A$ has exact rank $m \le \min(M,N)$. One can then write  $A = LQ$ (\cite{golub}), where $L$ is $M \times m$ lower triangular and $Q$ is $m \times N$ with
orthonormal rows. In particular, $L$ has a left inverse $(L^TL)^{-1}L^T$, and $Q$ has full rank.
Moreover, the null space of $A$ is spanned by the columns of an $N \times (N-m)$
matrix $U$ with orthonormal columns.

\smallskip
We are interested in \emph{minimal linearly dependent} sets of columns
of $A$, that is, subsets $\cJ\subset \overline{1,N}$ such that
$A(:,\cJ)$ does not have full rank, but any matrix $A(:,\cJ')$ with $\cJ'\subset \cJ, \, \cJ' \neq \cJ$ has full rank.
Borrowing the corresponding term
from matroid theory, such a set $\cJ$ or the set of columns indexed by it will be called a {\it circuit} in this paper. Recall that
in the engineering literature on testability, a circuit is called a
\emph{canonical ambiguity group}; cf \cite{manetti, stenbak}. If $\cJ$
is a circuit, there exists a $N \times 1$ null vector $z$ of $A$ (or
equivalently of $Q$) such that $supp(z) = \cJ$. Reversely, there exists a circuit $\cJ \subset supp(z)$ for any null vector $z$. Thus $z$ is a \emph{sparse null vector} for $A$ if $supp(z)$ is a circuit, and $z$ is also a \emph{sparse column vector} for $U$ in this case. Circuits containing only one column clearly must be columns of zeroes in $A$.

\smallskip
A matrix with linearly independent columns does not have any circuits. More generally,
a column $A(:,j)$ does not belong to any circuit of $A$, if this column does not belong to the space spanned by the columns $A(:,i),\, i \ne j$ or equivalently if the rank of the matrix drops if column $j$ is removed. An explicit way of identifying  columns that do not belong to circuits is given below in Lemma \ref{lmm_nocircuit}. We now characterize circuits of $A$.

\medskip
\begin{lemma}\label{lmm_char} Let $A = LQ$ where $Q$ has orthonormal rows, and let $U$ have orthonormal columns such that $QU=0$ and the columns of $U$ span the null space of $A$. Let $\cJ \subset \overline{1,N}$, and let $\cJ^c$ be the complement of $\cJ$ in $\overline{1,N}$.

(i) Then $A(:,\cJ)$ has a non-trivial null vector if and only if $U(\cJ^c,:)$ has a non-trivial null vector.

(ii) The following properties are equivalent:

a) $\cJ$ is a circuit for $A$.

b) $Q(:,\cJ)$ has a one-dimensional null space spanned by a $|\cJ|$ -
vector $w$ that does not vanish anywhere.

c) $U(\cJ^c,:)$ has a non-zero null vector, and for all $k \in \cJ$ and $\cK = \cJ^c
\cup \{k\}$, the matrix $U(\cK,:)$ has full rank.

d) There are $\cK \subset \overline{1,N}$ with $\cJ \subset \cK $ and an $N$-vector $y$ with $supp(y) = \cJ$ such that
$Q(:,\cK)$ has a one-dimensional null space generated by $y(\cK)$.

e) There is  $\cI \subset \overline{1,N}$ with $\cI\cap \cJ =
\emptyset$ such that $U(\cI,:)$ has a one-dimensional null space
generated by a vector $d$, and $supp(Ud)= \cJ$.

\end{lemma}

\begin{proof} To prove (i), observe that any null vector $v$ of $A$ is of the form $v = Ud$, and the map $v \mapsto d$ is an isomorphism between the null space of $A$ and $\R^{N-m}$. If $supp(v) \subset \cJ$ and $v \ne 0$, then $d \ne 0$ and $v(\cJ^c) = U(\cJ^c,:)d = 0$. Reversely, if $U(\cJ^c,:)w = 0$ for some $w \ne 0$, then $supp(Uw) \subset \cJ$ and $Uw$ is a non-trivial null vector of $A$.

\medskip
We now turn to (ii). First note that $A(:,\cK)$ and $Q(:,\cK)$ have the same
null space for any $\cK \subset \overline{1,N}$. Let us prove that a) implies b). If a) holds, the null space of $Q(:,\cJ)$ is non-trivial. All non-zero null vectors of $Q(:\cJ)$ must have support equal to $\cJ$, since otherwise $\cJ$ would not be minimal. If the null space of $Q(:,\cJ)$ had dimension larger than 1, a linear combination of two null vectors could be found that vanishes at an index in $\cJ$ but not everywhere. Thus
the null space a) of $Q(:,\cJ)$ must be one-dimensional, and b) follows. Reversely, if b) holds, then there exists a nontrivial null vector $v$ of $A(:,\cJ)$. A non-trivial null vector of $A(:,\cJ)$ with strictly smaller support would be linearly independent of $v$, which is not allowed if b) holds. Hence $\cJ$ is a circuit of $A$.

\smallskip To prove the equivalence of a) and c), let $\cJ$ be a circuit of $A$. Then $U(\cJ^c,:)$ has a non-trivial null vector by part (i). Now assume that  for $\cK$ as in the assumption, $U(\cK,:)$ does not have full
rank. By (i), $A(:,\cK^c$ does not have full rank and $\cK^c \subset \cJ, \, \cK^c \ne \cJ$, contradicting the minimality of $\cJ$. Reversely, assume c). By (i), $\cJ$ contains a circuit of $A$. If $\cJ$ were not minimal, we could find $\cJ' = \cJ - \{k\}$ such that $A(:,\cJ')$ does not have full rank. But then with $\cK = \cJ'^c = \cJ \cup \{k\}$, $U(\cK,:)$ will not have full rank by (i), contradicting c).

\smallskip
Clearly b) implies d) - just take $\cK = \cJ$. For the reverse
conclusion, just take $w = y(\cJ)$. Then $w$ does not vanish
anywhere and spans the null space of $Q(:,\cJ)$.

\smallskip
To prove that a) - d) together imply e), choose $\cI = \cJ^c$. By (i), there is a null vector
$d$ of $U(\cJ^c,:)$. Then clearly $supp(Ud) \subset \cJ$ for all such null vectors. If the
inclusion were strict, e.g. $(Ud)_k = 0$ with $k \in \cJ$, then
$U(\cJ^c \cup \{k\},:)$ would not have full rank, contradicting c). Also, if $U(\cJ^c,:)$ had two linearly independent null vectors $d$ and $d'$, then the support of a suitable linear combination of $Ud$ and $Ud'$ would be strictly contained in $\cJ$, again contradicting c). Reversely, if e) holds, then $U(\cJ^c,:)d = 0$ and hence $\cJ$ contains a circuit of $A$ by (i). If $\cJ$ were not a circuit of $A$, e.g. if we could find a circuit $\cJ' \subset \cJ, \, \cJ' \ne \cJ$ of $A$, we could find a null vector $w = Ud'$ of $A$ with $supp(w) = \cJ'$ and $U(\cJ'^c,:)d' = 0$, hence also $U(\cI,:)d' =0$. Now $w = Ud'$ and $x = Ud$ are linearly independent, since they do not have the same support. Hence also $d$ and $d'$ are linearly independent, contradicting the assumption that $U(\cI,:)$ has a one-dimensional null space. Therefore $\cJ$ is a circuit of $A$. So e) implies a). This completes the proof.
\end{proof}

With these preparations, a prototype algorithm for detecting
circuits can be described. It operates on the matrix $Q$.

\begin{algorithm}Let $A = LQ$ as above, where $Q$ is $m \times N$ and has rank $m$.

1. Choose a subset $\cK \subset \overline{1,N}$.
2. Determine a matrix $Y$ whose columns are a basis of the null
space of $Q(:,\cK)$.

3. If $Y$ has rank 1, then $\cJ = supp(Y)$ is a circuit of $A$.

4. If $Y$ is the empty matrix or has more than one column, choose
a different subset $\cK$ and repeat the procedure.
\end{algorithm}

The algorithm stops with a circuit due to  Lemma 2.1 d). The
questions then arise how to choose $\cK$, what the size of this
subset should be, whether one can do better if $Y$ has more than
one column, how to conclude that there (probably) is no circuit of a given
size, and so on.

\smallskip
There is also a version of this prototype algorithm that operates
on the matrix $U$. It was proposed in \cite{manetti}. This
algorithm stops after finding a circuit, due to Lemma 2.1 e). The
same questions about the choice of $\cI$ and alternate stopping
criteria arise.
\begin{algorithm} Let $U$ be a $N \times (N-m)$ matrix whose
columns span the null space of $A$.

1. Choose a subset $\cI \subset \overline{1,N}$.

2. Determine a matrix $Z$ whose columns are a basis of the null
space of $U(\cI,:)$.

3. If $Z$ has rank 1 then $\cJ = supp(UZ)$ is a circuit of $A$.

4. If $Z$ is the empty matrix or has more than one column, choose
a different subset $\cI$ and repeat the procedure.
\end{algorithm}

\medskip
Let us now assume that the last $m$ columns of $Q$ form an
invertible matrix. This is always possible after permuting
columns. Thus $Q = \left( Q_1, \, Q_2 \right) = Q_2 \left( Q^*,
\I_m \right)$, where $Q^* = Q_2^{-1}Q_1$ is $m \times (N-m)$, and therefore
\begin{equation} \label{eq_decomp}A = LQ = \tilde L\left(Q^*, \I_m \right)
\end{equation}
with $\tilde L = LQ_2$. This decomposition was exploited in \cite{starzyk} to find circuits.
Partitioning $U = \left(\begin{matrix} U_1 \\ U_2
\end{matrix}\right)$, where $U_1$ is $(N-m) \times (N-m)$ and
$U_2$ is $m \times (N-m)$, we see that $0 = QU = Q_2 \left( Q^*
U_1 + U_2 \right)$. Then $U_1$ must be invertible, since otherwise
$U$ could not have full rank. Write
\begin{equation} U^* = U_2U_1^{-1}, \, C = \left(\begin{matrix}
\I_{N-m} \\ U^* \end{matrix} \right), \, U = CU_1 \, ,
\end{equation}
 then it follows that
\[Q^* + U^* = 0  \quad \text{and} \quad AC = 0 .\]
The matrix $C$ (or rather the set of its columns) is commonly called a \emph{fundamental null basis}, see e.g. \cite{coleman_pothen1}.
It is now easy to give a version of Lemma 2.1 that uses only
properties of $Q^*$ (or $U^*$), i.e. that refers only to the fundamental null basis $C$. For any $\cK \subset
\overline{1,N}$, set $\cK_1 = \cK \cap \overline{1,N-m}, \, \cK_2 = \{i \in \overline{1,m} \big| i+N-m \in \cK
\}$, and $\cK_{2,c} = \{i \in \overline{1,m}
\big| i+N-m \notin \cK\}$.

\medskip
\begin{lemma} Let $\cJ \subset \overline{1,N}$, and assume that the factorization (\ref{eq_decomp}) holds. The following properties are equivalent:

a) $\cJ$ is a circuit for $A$.

b) $Q^*(\cJ_{2,c},\cJ_1)$ has a one-dimensional null space spanned by a
vector $w$ that does not vanish anywhere, and $Q^*(\cJ_2,\cJ_1)w$ does
not vanish anywhere.

c) There is an index set $\cK$ with $\cJ \subset \cK \subset \overline{1,N}$ such
that $Q^*(\cK_{2,c},\cK_1)$ has a one-dimensional null space generated
by a vector $w$, with
\[supp(w) = \{i \big| \cK_1(i) \in \cJ_1\}, \quad supp(Q^*(\cK_2,\cK_1)w) =
\{i \big| \cK_2(i) \in \cJ_2\}\, .\]
\end{lemma}

\begin{proof}
Let $Q = \left(Q^*,\I_{m}\right)$. Let $y$ be an $N$-vector, then
$Q(:,\cJ)y(\cJ) = 0$ if and only if $Q^*(\cJ_{2,c},\cJ_1)y(\cJ_1) =0$ and
$Q^*(\cJ_2,\cJ_1)y(\cJ_1) = - y(\cJ_2)$. By Lemma 2.1, statements a) and b)
are therefore equivalent.

\smallskip
To show that a) and c) are equivalent, one shows similarly that
property d) in Lemma 2.1 reduces to property c) in the present
situation.
\end{proof}

\smallskip
Columns of $A$ that do not belong to any circuit can be easily identified if the factorization (\ref{eq_decomp}) or equivalently a fundamental null basis $C = \left(\begin{matrix}
\I_{N-m} \\ U^* \end{matrix}\right)$ are given.

\medskip
\begin{lemma} \label{lmm_nocircuit} Let $A$ be given such that the factorization (\ref{eq_decomp}) holds with a full rank left factor $\tilde L$. Column $j$ of $A$ belongs to a circuit of $A$ if and only if either $1 \le j \le N-m$ or if $N-m+1 \le j \le N$ and row $j-N+m$ of $Q^*$ does not vanish identically.
\end{lemma}

\begin{proof} Note that row $j-N+m$ of the right factor $(Q^*,\I_m)$ contains a 1 in column $j$. Let $e_k$ denote the $k$-th standard unit vector. If $1 \le j \le N-m$, then column $j$ of $\tilde L^{-1}A$ satisfies
\[\tilde L^{-1}A(:,j) = Q^*(:,j) = \sum_{Q^*(\nu,j) \ne 0} Q^*(\nu,j) e_{\nu} =  \sum_{Q^*(\nu,j) \ne 0} Q^*(\nu,j) A(:,N-m+\nu) \, .
\]
Thus this column is a unique linear combination with non-vanishing coefficients of some of the last $m$ columns and therefore is in a circuit. If $N-m+1 \le j \le N$ and $Q^*(j-N+m,k) \ne 0$ for some $k$, then again
\[\tilde L^{-1}A(:,k) = Q^*(:,k) = \sum_{Q^*(\nu,k) \ne 0} Q^*(\nu,k) e_{\nu} =  \sum_{Q^*(\nu,k) \ne 0} Q^*(\nu,k) A(:,N-m+\nu) \, .
\]
The term corresponding to $\nu = j-N+m$ cannot be dropped, hence column $j$ belongs to a circuit of $A$.

\smallskip
Reversely, if $N-m+1 \le j \le N$ and if $Q^*(j-N+m,:) = 0$, then clearly the rank of $(Q^*,\I_m)$ drops if column $j$ is removed, since the remaining matrix will now have a row of zeroes. Hence column $j$ does not belong to a circuit of $A$.
\end{proof}

\section{Random Search for Circuits}
In this section, we give an \emph{Monte Carlo} algorithm for
finding circuits up to a certain size, say $n$. If there is no
such circuit, the algorithm always states this correctly. If there
is such a circuit, it is found with probability $1-\epsilon$ in
$K$ steps, where $K \approx |\log \epsilon|
\left(\frac{N}{m}\right)^n$ and in each step typically a null
vector of a submatrix of $Q$ or $U$ must be found.

\begin{algorithm} \label{rand_alg}Let $A = LQ$ be as above. Let $n$ (the desired circuit size) and $\epsilon > 0$ be given.

0. Set $p = 1$.

While $p > \epsilon$,

1. Choose a random subset $\cK \subset \overline{1,N}$ with $|\cK| = m+1$,
uniformly from all these subsets, independently from previous selections.

2. Determine a full rank matrix $Z$ whose columns span the null
space of $Q(:,\cK)$.

3.a) If $Z = w$ has rank 1, then set
\[\cJ = \{ \cK(i) | i \in supp(w)\}\, . \]
The set $\cJ$ is a circuit of $A$. If $|\cJ| \le n$, STOP and return $\cJ$.
Otherwise set $r = |\cK|$, replace $p$ with
\[p \left(1 - \frac{\binom{N-n}{r-n}}{\binom{N}{r}}\right)\,,\]
and go to step 1.

3.b) If $rank(Z) = l >1$, replace $\cK$ with a random subset $\tilde \cK \subset \cK$ such that
$|\tilde \cK| = |\cK| - l + 1$ , selected uniformly from all such subsets of $\cK$, and go to step 2.
\end{algorithm}

Since the matrix $Q(:,\cK)$ is $m \times (m+1)$, step 2 always finds a non-trivial matrix $Z$ in the first attempt. Should $rank(Z) = l$ be larger than 1, and step 3.c be carried out with a smaller $\cK$, then the new matrix $Q(:,\cK)$ still has a non-trivial null space, but strictly smaller dimensions. Therefore, the method eventually ends up in step 3.a, i.e. it finds a subset $\cK$ of size $r$ such that $Q(:,\cK)$ has a one-dimensional null space. It then inspects this subset to determine if it contains a circuit of the desired size. A single pass of the algorithm that ends in step 3.a) will be called at \emph{trial}. It may involve several computations of null space bases $Z$.

\smallskip
A given circuit $\cJ$ of size $n$ will be detected in step 3.a if $\cJ \subset \cK$. This happens with probability $\frac{\binom{N-n}{r-n}}{\binom{N}{r}}$ and fails to happen with probability $1 - \frac{\binom{N-n}{r-n}}{\binom{N}{r}}$, where $r = |\cK|$. At any stage of the algorithm, the value $p$ therefore is the probability that a fixed circuit $\cJ$ of size $n$ would not have been detected up to this stage. If this probability is very small (less than $\epsilon$), then we may be confident (with a confidence level of $1 - \epsilon$) that no such circuit exists; hence the termination criterion.

\smallskip
Let $\rho = \frac{m+1}{N}$ and $\delta = \frac{n-1}{N}$. If step 3.a is reached immediately (the generic case), then  $|\cK| = m+1$, and the probability that a given circuit $\cJ$ of size $n$ is detected equals
\begin{equation}
\label{eq_prob_1} \frac{\binom{N-n}{m+1-n}}{\binom{N}{m+1}} = \prod_{j=1}^n \frac{m-n+j+1}{N-n+j} \ge \left(\frac{m+1}{N}\cdot \frac{m-n+2}{N-n+1} \right)^{n/2} = \left( \frac{\rho (\rho - \delta)}{{1-\delta}} \right)^{n/2}.
\end{equation}
For small $\delta$, this is approximately equal to $\rho^n$. Therefore, the probability of not detecting a fixed circuit of size $n$ with $K$ independent choices of $\cK$ is approximately bounded by $(1-\rho^n)^K$. This is certainly smaller than a given $\epsilon$ if $K \ge \frac{-\log \epsilon}{\rho^n}$. If there is indeed a circuit $\cJ$ of size $n$, the expected number of trials until it is found is approximately $\rho^{-n}$. Consider in particular instances of the problem where $\rho = \frac{m+1}{N} \ge \rho_0$, where $\rho_0 > 0$ is given, and $n$ is fixed. As $N,\, m \to \infty$, the probability that a given circuit of size $n$ is detected in a single trial is bounded below asymptotically by $\rho_0^{-n}$. Therefore, the expected number of steps to find a circuit of size $n$ is exponential in $n$, but does not depend directly on the problem size $Nm$. One can expect that circuits of moderate size $n$ are found rapidly if the rank defect $N-m$ is small relative to $N$.

\smallskip
If there is more than one circuit of size $\le n$, it will take fewer trials to find one of them. Suppose there are $k$ circuits of size $c_j$, $j = 1, \dots, k$, and assume that the circuits are all disjoint and their sizes are small compared to $N$. The probability of selecting a set $\cK$ of size $m+1$ that contains a specific circuit of size $c_j$ then is $p_j \approx \rho^{c_j}$. The probability of selecting a $\cK$ that contains one of these circuits is approximately
\begin{equation} \label{eq_prob_2} 1 - \prod_{j=1}^k (1 - \rho^{c_j}) \, .
\end{equation}
The reciprocal of this number is (close to) the expected number of trials until a circuit is found.

\smallskip
We now give a version of this algorithm that operates on the matrix
$Q^*$, where  $A = L_1(Q^*,\I_m)$, or equivalently on the fundamental null matrix $ C$. As was noted in the previous section, this may require permuting the columns of $A$.

\begin{algorithm} Let $A = L_1(Q^*,\I_m)$ be as above.
Let $n$ (the desired circuit size) and $\epsilon > 0$ be given.

0. Set $p = 1$.

While $p > \epsilon$,

1. Choose a random subset $\cK \subset \overline{1,N}$ with $|\cK| = m+1$,
uniformly from all these subsets, independently from previous selections. Set $\cK_1 = \cK \cap \overline{1,N-m}, \, \cK_2 = \{i \in \overline{1,m} \big| i+N-m \in \cK
\}$, and $ \cK_{2,c} = \{i \in \overline{1,m} \big| i+N-m \notin \cK\}$ .

2. Determine a full rank matrix $Z$ whose columns span the null
space of $Q^*(\cK_{2,c},\cK_1)$.

3.a) If $Z = w$ has rank 1, then set
\[\cJ = \{ \cK_1(i) | i \in supp(w)
\} \cup  \{\cK_2(i)+N-m | i \in supp\left(Q^*(\cK_2,\cK_1)w\right)\}. \]
The set $\cJ$ is a circuit of $A$. If $|\cJ| \le n$, STOP and return $\cJ$. Otherwise set $r = |\cK|$, replace $p$ with
\[p \left(1 - \frac{\binom{N-n}{r}}{\binom{N}{r}}\right)\,,\]
and go to step 1.

3.b) If $rank(Z) = l >1$, replace $\cK$ with a random subset $\tilde \cK \subset \cK$ such that
$|\tilde \cK| = |\cK| - l + 1$ , selected uniformly from all such subsets of $\cK$, and go to step 2.
\end{algorithm}

After choosing $\cK$ in step 1, the matrix $Q^*(\cK_{2,c},\cK_1)$ is $k \times (k+1)$, where $k = |\cK_{2,c}| \le \min (m, N-m)$, and therefore step 2 always finds a non-trivial matrix $Z$ in the first attempt. Should $rank(Z) = l$ be larger than 1, and step 3.b be carried out, then the new matrix $Q^*(\cK_{2,c},\cK_1)$ is obtained from the previous one by deleting $s$ columns and adding $l-1-s$ rows, where $0 \le r \le l-1$. As a result, this matrix still has a non-trivial null space, but strictly smaller dimensions. Therefore, the method eventually ends up in step 3.a, i.e. it finds a subset $\cK$ of size $r$ such that $Q^*(\cK_{2,c},\cK_1)$ has a one-dimensional null space. It then inspects this subset to determine if it contains a circuit of the desired size.

\smallskip
If $\cK$ is selected randomly such that $|\cK| = m+1$, then $Q^*(\cK_{2,c},\cK_1)$ is $k \times (k+1)$ with $k\le m, \, k \le N-m$, and $k$ has a hypergeometric distribution with expected value $\frac{(m+1)(N-m)}{N} = (m+1)\left(1 - \frac{m}{N}\right) \approx N \rho (1-\rho) $. The expected computational effort to find a basis of the null space of such a matrix is proportional to $N^3 \rho^3 (1-\rho)^3$, to leading order (\cite{golub}). This implies that considerable savings are achieved by precomputing $A = L_1(Q^*,\I_m)$ if the rank $m$ or the rank defect $N-m$ are small relative to $N$, since then $\rho \ll 1$ or $1-\rho \ll1$.

\section{Excluding Circuits of a Certain Size}
In this section, a deterministic algorithm will be given that
allows one to conclude with certainty that a matrix $A$ does not
have a circuit of size $n$. This is a derandomized version of the \emph{Monte
Carlo} algorithm of the previous section. It is based on the following observation.

\medskip
\begin{lemma}\label{lmm_subsets} Let $A$ be an $M \times N$ matrix of rank $m$. Let $\cJ_1, \dots, \cJ_r \subset \overline{1,N}$ be disjoint and non-empty such that $\bigcup_j \cJ_j = \overline{1,N}$. For $\cC \subset \overline{1,r}$, set $\cJ(\cC) = \bigcup_{j \in \cC} \cJ_j$. Assume that $A$ has a circuit of size $n$. Then there exists  $\cC$ with $|\cC| = n$ such that $A(:,\cJ(\cC))$ has a circuit of size $n$.
\end{lemma}

\smallskip
\begin{proof}If $A$ has a circuit $\cI$ of size $n$, we can set $\tilde \cC = \{j \, | \, \cJ_j \cap \cI \ne \emptyset\}$. This set has at most $n$ elements, and $\cI \subset \cJ(\tilde\cC)$. Enlarging $\tilde \cC$ if necessary, we obtain a set $\cC \subset \overline{1,r}$ with $n$ elements such that $\cI \subset \cJ(\cC)$. A vector from the null space of $A$ that is supported on $\cI$ may be restricted to $\cJ(\cC)$, resulting in a vector in the null space of $A(:,\cJ(\cC))$ with the same support. Hence also $A(:,\cJ(\cC))$ has a circuit of size $n$.
\end{proof}

\smallskip
This observation is the basis of a recursive algorithm to determine if a given matrix $A$ has a circuit of size $n$ or smaller. The algorithm returns the value $true$ if there is a circuit of size at most $n$ and $false$ otherwise.

\smallskip
\begin{algorithm} \label{syst_alg} Let $A$ be an $M \times N$ with $rank(A) = m < \min(M,N)$. Compute a logical variable $\alpha = circuitfind(A,n)$ as follows.

0. Compute $A = LQ$ such that $Q$ is $m \times N$. Find the smallest integer $r$ such $n \lceil \frac{N}{r} \rceil \le m+1$ and set $k = \lfloor \frac{N}{r}\rfloor$. Find disjoints subsets $\cJ_1, \dots, \cJ_r \subset \overline{1,N}$, all of size $k$ or $k+1$. Set $\cC = \overline{1,n} \subset \overline{1,r}$ and set $\alpha = false$.

1. While $\cC \ne \overline{r-n+1,r}$ and $\alpha = false$, do the following:

\smallskip
1.a) Set $\cJ(\cC) = \bigcup_{j \in \cC} \cJ_j$ and find the dimension $d$ of the null space of $Q(:,\cJ(\cC))$.

\smallskip
If $d = 0$, replace $\cC$ with the next $n$-element subset of $\overline{1,r}$, in lexicographical order and return to 1.

\smallskip
If $d=1$, find a non-trivial null vector $z$ of $Q(:,\cJ(\cC))$. If $supp(z)$ has no more than $n$ elements, STOP and return $\alpha = true$. Otherwise replace $\cC$ with the next $n$-element subset of $\overline{1,r}$, in lexicographical order and return to 1.

\smallskip
If $d>1$, set $\alpha = circuitfind(Q(:,\cJ(\cC)),n)$, computed with this algorithm. If $\alpha = true$, STOP. Otherwise replace $\cC$ with the next $n$-element subset of $\overline{1,r}$, in lexicographical order and return to 1.
\end{algorithm}

\medskip
Since $m+1 \le N$, the integer $r$ computed in step 0 is never smaller than $n$. In fact $n = r$ is only possible if $N=m+1$. In this case, the only possible circuit of $A$ is the common support of the vectors in the one-dimensional null space of $A$. Therefore, one obtains $r > n$ except in this trivial case. Then there are $\binom{r}{n}$ possible matrices $Q(:,\cJ(\cC))$ which may be examined in lexicographic order, starting with $\cC = \{1, \dots, n\}$ and ending with $\cC = \{r-n+1, \dots, r\}$. If $A$ has a circuit of size $n$ or smaller, one of these matrices must also have such a circuit, by Lemma \ref{lmm_subsets}. All these matrices have $m$ rows and at most $m+1$ columns, thus the dimension $d$ of their null spaces may range from $d=0$ to $d = m$. If $d=0$, the matrix $Q(:,\cJ(\cC))$ does not have a circuit, and the algorithm examines the next subset $\cJ(\cC)$. If $d=1$, the single possible circuit of this matrix consists of the common support of the vectors in this null space, by Lemma \ref{lmm_char}. The algorithm examines this circuit and stops with a value $\alpha = true$ if the circuit has size at most $n$. Finally, if $d>1$, there may still be a circuit of size $n$ or smaller in $Q(:,\cJ(\cC))$, and it can be found by applying the same algorithm. This matrix has strictly fewer columns than $Q$, and its rank is also strictly smaller than the rank of $Q$. Hence the recursion will terminate after finitely many calls.

\smallskip
Using as before the notation $\rho = \frac{m+1}{N}$, we see that $r \approx n \rho^{-1}$. For the class of instances where  $\frac{m}{N}$ remains bounded away from $0$, the number of matrices that has to be examined to exclude the presence of circuits up to size $n$ is exponential in $n$, but it is independent of the problem size $Nm$, assuming of course that the algorithm does not call itself. Specifically, consider instances where $\frac{m+1}{N} \ge \rho_0$ and $n \le \frac{s\rho_0}{1-\rho_0}$ for some integer $s$ and some $\rho_0$. Then $n \frac{N}{n+s} \le m+1$, and we may take $r=n+s\le (1-\rho_0)^{-1}s$. In this case, the number of matrices that has to be examined is approximately bounded by
\[
\binom{n+s}{n} \le \binom{(1-\rho_0)^{-1}s}{s}  \sim \left( 2 \pi \rho_0 s\right)^{-1/2} A^s, \quad A = (1-\rho_0)^{\rho_0-1}\rho_0^{-\rho_0} > 1
\]
by Stirling's formula. Recall that for this class of instances, the expected number of matrix examinations to find a circuit of length $n$ with the randomized algorithm of the previous section is also bounded independently of $Nm$  and exponential in $n$.

\medskip
As before, it is possible to compute the factorization $A = L_1 (Q^*,\I_m)$ and achieve additional computational savings. We leave the details to the reader.

\section{Inexact Circuits}

We now discuss situations in which the $M \times N$ matrix $A$ has full rank and we wish to find a small subset of columns for which a linear combination vanishes approximately. Let $\epsilon > 0$ and $x$ be an $N$-vector with $\|x\| = 1$, then we say that $\cI \subset \overline{1,N}$ is an \emph{$\epsilon$-near circuit with witness vector $x$} if $supp(x) = \cI, \, \|Ax \| \le \epsilon$, and all singular values of $A(:,\cI_1)$ are larger than $\epsilon$ whenever $\cI_1 \subset \cI, \, \cI_1 \ne \cI$. We are interested in circuits with "few" elements (relative to the number of columns  $N$). Thus if a candidate for an $\epsilon$-near circuit with corresponding witness vector has been proposed, it is an easy matter to verify this. In particular, finding all singular values of $A(:,\cI)$ is a matter of $O(M |\cI|^2)$ operations, and it is possible to find all singular values of all matrices $A(:,\cI_1)$ with $\cI_1 = \cI - \{k\}$ for some $k$ in $O(M|\cI|^3)$ operations or faster, if suitable downdating methods are used.

\medskip
\begin{lemma} \label{lmm_near} Let  $\epsilon > 0$ and let $A$ be  an $M \times N$ matrix with singular values $0 \le \sigma_1 \le  \sigma_2 \le \dots \le \sigma_{max}$ and singular value decomposition $A = U S V^T$. Let $V = (v, V_2)$, where $v$ is the first column of $V$, corresponding to $\sigma_1$.  Let $\cI \subset \overline{1,N}$.

a) If $\cI$ is an $\epsilon$-near circuit of $A$ with witness vector $x$ and $\epsilon < \sigma_2$, then
\[\sum_{j \notin \cI}  v_j^2 \le \frac{\epsilon^2}{\sigma_2^2 - \epsilon^2}.
 \]

b)If
\[\sum_{j \notin \cI}  v_j^2 = \delta^2
\]
then $\cI$ is an $\epsilon'$-near circuit with
\[\epsilon' \le \sqrt{\sigma_1^2(1-\delta^2) + \sigma_{max}^2 \delta^2}.\]

\end{lemma}

\begin{proof}To prove part a),set $x_1 = vv ^Tx$ and $x_2 = x-x_1$. Then clearly $V_2^Tx_1 = 0$ and $1 = \|x_1\|^2 + \|x_2\|^2$. Also,
\[\epsilon^2 \ge \|Ax\|^2 = \|Ax_1\|^2 + \|Ax_2\|^2 \ge \|Ax_2\|^2 \ge \sigma_2^2\|x_2\|^2  \, .
\]
Consequently
\[\frac{\epsilon^2}{\sigma_2^2}\ge \|x_2\|^2 \ge \sum_{j \notin \cI} x_{2,j}^2 = \sum_{j \notin \cI} x_{1,j}^2 \quad \text{and} \quad \|x_1\|^2 \ge 1 - \frac{\epsilon^2}{\sigma_2^2} \, .\]
Therefore $v =  \|x_1\|^{-1}x_1$ has the desired properties.

 \smallskip
 To prove part b), let $z_j = v_j$ if $j \in \cI$ and $z_j = 0$ otherwise, and set $x = \|z\|^{-1}z = (1 - \delta^2)^{-1}z$. The cosine of the angle enclosed by $v$ and $z$ is $\sqrt{1-\delta^2}$, and hence $x = \sqrt{1-\delta^2}v + \delta y$ where $\|y\| = 1$ and $(Ay)^TAv = 0$; therefore $\|Ay\| \le \sigma_{max}$. Hence
 \[\|Ax\|^2 = (1-\delta^2)\|Av\|^2 + \delta^2\|Ay\|^2 \le \sigma_1^2(1-\delta^2) + \sigma_{max}^2 \delta^2 \, .
 \]
\end{proof}

\smallskip
This observation suggests that one should look for $\epsilon$-near circuits by selecting subsets $\cK \subset \overline{1,N}$ such that $A(:,\cK)$ has only one singular value  $\sigma_1 < \epsilon$. A near circuit with a witness vector may then be discovered by setting all small components of the corresponding singular vector $v$ equal to zero.This is the idea of the following algorithm. It either produces a candidate set $\cI$ for an $\epsilon$-near circuit, or it returns the answer that no such near circuit exists.

\smallskip
\begin{algorithm} \label{near_rand_alg}Let $A$ be as above, with $M \le N$. Let the desired circuit size $n$, $\delta  \in (0,1)$,  and $\epsilon > 0$ be given.

0. Set $p = 1$. Determine $m$ such that $\epsilon$ separates the $m$ largest singular values of $A$ from the $M-m$ smallest singular values.

While $p > \delta$,

1. Choose a random subset $\cK \subset \overline{1,N}$ with $|\cK| = m+1$,
uniformly from all these subsets, independently from previous selections.

2. Find the singular value decomposition $USV^T = A(:,\cK)$, with singular values $\sigma_1 \le \sigma_2 \le \dots$.

3.a) If $\sigma_1 \le \epsilon \le \sigma_2$, consider $v$, the first column of $V$. Set
$\cJ$ equal to the support of the $n$ largest entries of $v$. Compute the smallest singular value of
$A(:,\cK(\cJ))$. If this value is less than $\epsilon$, STOP and return $\cI = \cK(\cJ)$.
Otherwise replace $p$ with
\[p \left(1 - \frac{\binom{N-n}{|\cK|-n}}{\binom{N}{|\cK|}}\right)\,,\]
and go to step 1.

3.b) If $\sigma_1 \le \sigma_2 \le \dots \le \sigma_l \le \epsilon < \sigma_{l+1}$, replace $\cK$ with a random subset $\tilde \cK \subset \cK$ with $l-1$ fewer elements, selected uniformly from all such subsets of $\cK$, and go to step 2.
\end{algorithm}

\smallskip
Since the matrix $A(:,\cK)$ has $m+1$ columns, at least one of its singular values is less or equal than $\epsilon$, by well-known interlacing properties of singular values (\cite{golub}). Hence either case 3.a) or case 3.b) always occurs. Should there be more than one singular value that is less or equal than $\epsilon$, then $\cK$ is replaced with a smaller $\cK$, and the smaller matrix $A(:,\cK)$ still has at least one singular value less or equal than $\epsilon$. Therefore, the method eventually ends up in step 3.a. By Lemma \ref{lmm_near}, the singular vector corresponding to this singular value should suggest a near circuit. The method therefore inspects this singular vector to determine a candidate for a circuit of the desired size.

\smallskip
As before, if no near circuit of size up to $n$ is found before $p < \delta$, we may be confident with confidence $1 - \delta$ that no near circuit of the desired size exists, that is, the search has been performed sufficiently often such that a near circuit with the desired properties would have been found with probability at least $1 - \delta$, if it existed. The number of trials until a near circuit is found has the same distribution as the corresponding quantity in algorithm \ref{rand_alg}. If a candidate $\cI$ has been found, it should still be tested; that is, the smallest singular value $\sigma(\cI)$ of $A(:,\cI)$ should be found as well as the minimal eigenvalues $\sigma(\cI')$ of all matrices $A(:,\cI')$ with $\cI'= \cI - \{k\}$. If $\sigma(\cI) \le \epsilon < \sigma(\cI')$ for all $i$, an $\epsilon$-near circuit has been found.   

\section{Practical Considerations and Numerical Examples}

When looking for circuits, one should first identify and remove those columns that cannot belong to any circuits. This is easily done using Lemma \ref{lmm_nocircuit} and reduces $N$ and $m=rank(A)$ by the same fixed amount. Next, one should use a version of the random search algorithm \ref{rand_alg}, starting with small circuit sizes. A repeated random search will reveal whether there is more than one circuit present. Starting with algorithm \ref{syst_alg} to find circuits is also possible but not recommended, since it may lead to very long run times for reasons that will be explained below.

\smallskip
The main computational step in algorithm \ref{rand_alg} is the determination of $Z$, a matrix whose columns span the null space of $A$, in step 2. This has to be done repeatedly until $rank(Z) = 1$. One would expects that $rank(Z) = 1$ already when the first random subset $\cK$ of size $m+1$ is selected in step 1 of that algorithm. However, if the rank $m$ of $A$ is small relative to the number of columns $N$ (e.g. $\rho = \frac{m}{N} \approx 0.3$ or smaller), the null space dimension $rank(Z)$ is typically larger than 1 for the first selection of $\cK$, and $Z$ must be computed again for smaller subsets of columns. This is just the computation of a null space of a downdated submatrix, and can be thus be done rapidly, see \cite{golub}.

\smallskip
A numerical experiment was carried out to test if the probability of detecting a circuit in a single trial depends only on the ratio of the matrix rank $m$ to the number of columns $N$ and the size of the circuits that are present. Random matrices with $N=100$ columns and $m = \rho N$ rows, with $\rho = \frac{m}{N} \in \{0.3, 0.5, 0.7, 0.9\}$ were generated that had circuits with sizes given by vectors of integers $C = (c_1,\dots, c_k)$. This was done by first generating an $m \times (N-k)$ matrix $A'$ with independent entries drawn from a standard normal distribution, choosing $k$ random sets of columns of size $c_1-1,\dots, c_k-1$ of $A'$ and forming $k$ random linear combinations from them, and appending these $k$ vectors to the matrix $A'$ to form $A$. The probability of detecting a circuit in a single trial was estimated from 1000 trials applied to the same fixed matrix $A$. The results are given in the table below. The last column of the table contains the approximate probability computed from (\ref{eq_prob_2}).

\bigskip
\begin{quote}
\begin{tabular}{|r|r|r|r|r|r|r|}
  \hline
   $\rho$&$C$& $N=100$ & $N=200$ & $N=400$ & $N=800$ & \text{Theory} \\
  \hline
  .9&(6) & .55 & .56 & .52 &  .55 & .53 \\
  .7& (5,5,5) & .42 & .44 & .43 & .42 & .42 \\
  .5& (4,4) & .15 & .11 & .11 & .13 & .12 \\
  .3& (3,3,3,3,3) & .13 & .14 & .14 & .10 & .13 \\
  \hline
\end{tabular}
\end{quote}

\bigskip
The table shows nearly constant detection probabilities, independent of the number of columns $N$, close to the approximate values in the last column. While the actual detection probability for a single trial of course does depend on the matrix, the experiment confirms the results of the discussion in section 3: The detection probability for a single trial depends essentially only on $\rho = \frac{m}{N}$ and on the number and sizes of circuits that are present.

\smallskip
The systematic algorithm \ref{syst_alg} examines a fixed set of $\binom{r}{n}$ submatrices of $A$ for the presence of circuits. Unless something is known about the likely location of a circuit, the subsets $\cJ_1,\dots, \cJ_r$ in step 0  should be generated randomly.  The number of submatrices to be examined is easy to determine at the start, but it may increase during the execution, since the algorithm is recursive. Thus the computational effort may increase substantially beyond the initial estimate. The algorithm can also be used to search for a circuit.

\smallskip
In an experiment, a $m \times N$ random matrix with $N=100$ columns and $m = \rho N$ orthonormal rows was prepared that contained a single circuit of size $c=5$. Algorithms \ref{rand_alg} and \ref{syst_alg} were used to find this circuit. The ordering of the columns was permuted randomly between attempts. The mean number of times that a nullspace had to be computed until the circuit was detected was estimated from 100 attempts for both algorithms. Also recorded is the expected number of trials from formula \ref{eq_prob_1}. The results are given in the table below.

\bigskip
\begin{quote}
\begin{tabular}{|r|r|r|r|r|}
  \hline
   & $\rho=.9$ & $\rho=.7$ & $\rho=.5$ & $\rho=.3$ \\
   \hline
  Algorithm \ref{rand_alg} & 1.61 & 5.48 & 28.7 & 512 \\
  Algorithm \ref{syst_alg} & 2.18 & 14.9 & 66.8 & 2270 \\
  Expected trials & 1.69 & 5.95 & 32 & 412 \\
  \hline
\end{tabular}
\end{quote}

\bigskip
The table shows that the observed average number of nullspace evaluations tracks the expected number of trials closely if random search is used (algorithm \ref{rand_alg}) and also for the case of systematic search (algorithm \ref{syst_alg}) for $\rho$ close to 1. For smaller $\rho$, the systematic search algorithm tends to require substantially more nullspace evaluations than the random method, no doubt because of the recursion.

\smallskip
The systematic search algorithm \ref{syst_alg} could also be used to find all circuits of a given size. It turns out that any given circuit will typically be detected many times, if this is attempted, leading to very long execution times. This occurs independently of self-calls of this algorithm. Hence one should only turn to algorithm \ref{syst_alg} if the absence of such circuits is suspected with high confidence, after repeated random searches have turned up nothing.

\smallskip
Unlike circuits of a given size, near circuits as defined in section 5 are not unique. Hence algorithm \ref{near_rand_alg} does not behave as predictably as algorithm \ref{rand_alg}. In particular, if $\epsilon$ is chosen too large in this algorithm, then there may be many $\epsilon$-near circuits which can be detected, while there are none if $\epsilon$ is too small. Using a bisection approach, it is possible to find $\epsilon$-near circuits with near minimal $\epsilon$ fairly reliably.

\smallskip
Given a matrix $A$ (generated at random), it is observed that the smallest singular values of randomly selected submatrices with a fixed number of columns are very nearly normally distributed $\sim N(\mu,\sigma)$ with $\mu$ and $\sigma$ depending on the matrix. Hence randomly selected submatrices will very rarely have smallest singular value less than $\mu - 4 \sigma$ or so. In this situation, Algorithm \ref{near_rand_alg} is capable of detecting submatrices for which the smallest singular value is less than $\mu - 8 \sigma$, with high reliability.

\section{Other Problems}

There are a number of other problems which can be addressed with modifications of the methods in this note. We only discuss problems that may be posed for general matrices; that is, problems related to circuits in specific matrices coming from graph theory, electrical engineering, or statistical applications will not be discussed.

\smallskip
A simple variation on the task of finding one circuit of a given size is to find all circuits up to a given size. Clearly this may be attacked by repeated application of algorithm \ref{rand_alg}. Since circuits with fewer columns are more likely to be detected with this method, it may be necessary to delete a column in a circuit that has already been found from the matrix in order to find specifically longer circuits that do not contain this column.

\smallskip
A common problem is to find circuits or near circuits that have prescribed intersection properties with a given collection of sets of columns. This occurs e.g. in statistics, where a portion of the variables may be thought of as predictors and another one as responses. Circuits that contain one response and a small number predictors are of special interest in problems of model selection. Algorithms \ref{rand_alg} and \ref{syst_alg} are easily modified to handle such situations.

\smallskip
One may be interested in sampling randomly from the set of all $\epsilon$-near circuits up to a given size, for a given $\epsilon$. Algorithm \ref{syst_alg} effectively provides such a random sampling scheme. However, it is unclear what the sampling distribution is in this case, and it appears to be difficult and laborious to estimate its properties.

\smallskip
We finally mention the following update problem: Suppose $A = \left(\begin{matrix} A_1 \\ A_2
\end{matrix} \right)$ and a set of circuits of $A_1$ has been identified. Is it possible to exploit this information to speed up the detection of circuits of the full matrix $A$? An extreme version of this task, related to subspace tracking, occurs if rows are added to $A$ one at a time and a set of circuits has to be maintained or modified. Progress on such incremental algorithms has been made in \cite{neylon}.

\end{document}